\documentclass[smallextended]{svjour3} 
\smartqed 
\usepackage{amsmath,bm}
\usepackage{amssymb}
\usepackage{graphicx}
\usepackage{cite,multirow}
\usepackage{mwe}
\usepackage{etoolbox}
\usepackage{booktabs}
\usepackage{cite}
\usepackage{subfig}
\usepackage{caption}
\usepackage{xcolor}
\usepackage{eucal}
\usepackage{setspace}
\usepackage[colorlinks=true,citecolor=magenta]{hyperref}
\usepackage{algorithm}
\usepackage[noend]{algpseudocode}
\algnewcommand\algorithmicinput{\textbf{Input:}}
\algnewcommand\Input{\item[\algorithmicinput]}

\algnewcommand\algorithmicoutput{\textbf{Output:}}
\algnewcommand\Output{\item[\algorithmicoutput]}

\usepackage{mathabx}
\usepackage{tabularx,ragged2e,booktabs,caption}
\newcolumntype{C}[1]{>{\Left}m{#1}}

\newcommand{\fr}[1]{\mathfrak{#1}}

\allowdisplaybreaks
\usepackage{float}
\floatstyle{ruled}
\newfloat{model}{thp}{lop}
\floatname{model}{Model}

\usepackage{bm}
\allowdisplaybreaks

\makeatletter
\patchcmd{\@makecaption}
  {\scshape}
  {}
  {}
  {}
\makeatother

\newtheorem{thm}{Theorem}

\newtheorem{lem}[thm]{Lemma}

\begin{document}
\title{Optimization-Based Bound Tightening using a Strengthened QC-Relaxation of the Optimal Power Flow Problem}

\author{Kaarthik Sundar\and 
Harsha Nagarajan \and 
Sidhant Misra \and
Mowen Lu \and
Carleton Coffrin \and
Russell Bent
}
\institute{K. Sundar, C. Coffrin \at 
              Information Systems and Modeling (A-1), Los Alamos National Laboratory, NM, USA 
            \and 
            H. Nagarajan, S. Misra, R. Bent \at
              Applied Mathematics and Plasma Physics group (T-5), Los Alamos National Laboratory, NM, USA 
           \and
        M. Lu \at 
        Department of Industrial Engineering, Clemson University, SC, USA
}

\date{Received: date / Accepted: date}

\maketitle

\begin{abstract}
This article develops a strengthened convex quadratic convex (QC) relaxation of the AC Optimal Power Flow (AC-OPF) problem and presents an optimization-based bound-tightening (OBBT) algorithm to compute tight, feasible bounds on the voltage magnitude variables for each bus and the phase angle difference variables for each branch in the network. Theoretical properties of the strengthened QC relaxation that show its dominance over the other variants of the QC relaxation studied in the literature are also derived. The effectiveness of the strengthened QC relaxation is corroborated via extensive numerical results on benchmark AC-OPF test networks.  In particular, the results demonstrate that the proposed relaxation consistently provides the tightest variable bounds and optimality gaps with negligible impacts on runtime performance. 
\end{abstract}


\section*{Nomenclature}
\noindent
\textbf{Sets and Parameters} \\
{$\cal{N}$} - {set of nodes (buses)} \\
{$\cal{G}$} - {set of generators} \\
{${\cal{G}}_i$} - {set of generators at bus $i$} \\
{$\cal{E}$} - {set of \textit{from} edges (branches)} \\
{$\cal{E}^{R}$} - {set of \textit{to} edges (branches) } \\
{$\bm{c}_0, \bm{c}_1, \bm{c}_2$} - {generation cost coefficients} \\
{$\bm{i}$} - {imaginary number constant} \\
{$\bm{Y}_{ij} = \bm{g}_{ij} +\bm{i}\bm{b}_{ij}$} - {admittance on branch $ij$}\\
{$\bm{S}_{i}^{\bm{d}} = \bm{p}_{i}^{\bm{d}} + \bm{i}\bm{q}_{i}^{\bm{d}}$} - {AC power demand at bus $i$} \\
{$\bm{s}_{ij}^{\bm u}$} - {apparent power limit on branch $ij$} \\
{$\bm{\theta}^{\bm l}_{ij},\bm{\theta}^{\bm u}_{ij}$} - {phase angle difference limits on branch $ij$} \\
{$\bm{\theta}_{ij}^{\bm{m}}$} - {$\max(|\bm{\theta}^{\bm l}_{ij}|,|\bm{\theta}^{\bm u}_{ij}|)$ on branch $ij$} \\
{$\bm{v}^{\bm l}_i, \bm{v}^{\bm u}_i$} - {voltage magnitude limit at bus $i$} \\
{$\bm{S}_{i}^{\bm{gl}}$, $\bm{S}_{i}^{\bm{gu}}$} - {power generation limit at bus $i$} \\
{$\fr{R}(\cdot)$} - {real part of a complex number} \\
{$\fr{I}(\cdot)$} - {imaginary part of a complex number} \\
{$(\cdot)^*$} - {hermitian conjugate of a complex number} \\
{$\mid\cdot\mid, \angle\cdot$} - {magnitude, angle of a complex number} \\

\noindent
\textbf{Continuous variables} \\
{${V_i} = v_i\bm{e}^{\bm{i}\theta_i}$} - {AC voltage at bus $i$} \\
{$\theta_{ij} = \angle V_i - \angle V_j$} - {phase angle difference on branch $ij$} \\
{$W_{ij}$} - {AC voltage product on branch $ij$, i.e., $V_iV^*_j$} \\
{$S_{ij} = p_{ij} +\bm{i}q_{ij}$} - {AC power flow on branch $ij$} \\
{$S^g_i= p^g_i+\bm{i}q^g_i$} - {AC power generation at bus $i$} \\
{$l_{ij}$} - {current magnitude squared on branch $ij$} \\

\textbf{Notation} In this paper, constants are typeset in bold face. In the AC power flow equations, the primitives, $V_i$, $S_{ij}$, $S^g_i$, $\bm{S}_{i}^{\bm{d}}$ and $\bm{Y}_{ij}$ are complex quantities. Given any two complex numbers (variables/constants) $z_1$ and $z_2$, $z_1 \geqslant z_2$ implies $\fr{R}({z}_1) \geqslant \fr{R}({z}_2)$ and $\fr{I}({z}_1) \geqslant \fr{I}({z}_2)$. $|\cdot|$ represents absolute value when applied to a real number.

\section{Introduction} \label{sec:intro}
The AC Optimal Power Flow (AC-OPF) problem is one of the most fundamental optimization problems for economic and reliable operation of the electric transmission system. Since its introduction in 1962 \cite{Carpentier1962}, efficient solution techniques to solve the AC-OPF have garnered a lot of attention from the research community. The objective of the AC-OPF is to minimize the generation cost while satisfying the power flow constraints and the network limits. The fundamental difficultly with solving the AC-OPF arises due to the nonlinear and non-convex nature of the power flow constraints. The literature with regards to AC-OPF can predominantly be classified into the one of the following three groups: (i) developing fast algorithms to compute a local optimal solution to the AC-OPF either using meta-heuristics or numeral techniques like gradient descent \cite{Frank2012} etc., (ii) developing convex relaxations that convexify the feasible space defined by the AC-OPF, and (iii) developing global optimization algorithms for AC-OPF \cite{Kocuk2016}. The NP-hardness of AC-OPF \cite{Lehmann2016} makes guarantees on feasibility and global optimality very difficult and hence, the past decade has seen a surge in the work devoted towards developing convex relaxations of AC-OPF. They include the Semi-Definite Programming (SDP) \cite{Bai2008}, Second Order Cone (SOC) \cite{Jabr2006}, the recent Quadratic Convex (QC) \cite{Hijazi2017} and Moment-Based \cite{Molzahn2014} relaxations. In general, convex relaxations of AC-OPF are appealing because they can provide lower bounds to the AC-OPF objective value, prove infeasibility of the AC-OPF, or can aid in proving global optimality by producing a feasible solution in the non-convex space defined by the AC-OPF. One major factor that parameterizes the strength of the convex relaxation for the AC-OPF is the variable bounds. This dependence goes both ways, i.e., tightened bounds can aid in providing tighter convex relaxations and better convex relaxations can aid in tightening the variable bounds even further \cite{Coffrin2015,Coffrin2018}. In this article, we exploit this dependence between bound-tightening and strong convex relaxations in two novel ways: (i) we first present a strengthened QC relaxation that uses an extreme-point representation and strictly dominates the state-of-the-art QC relaxations in the literature \cite{Hijazi2017,Lu2018} and (ii) we develop an optimality-based bound-tightening algorithm (OBBT) that exploits the strengthened QC relaxations. These two novel contributions are put together to obtain lower bounds, that are better than the current known lower bounds, for the benchmark AC-OPF problem instances. We also note that variants of the bound-tightening algorithm presented in this article are used routinely in the mixed-integer nonlinear programming literature \cite{Nagarajan2019,Puranik2017} and also in algorithmic approaches
used to tighten variable bound in AC-OPF \cite{Chen2016,Coffrin2015}. 
Furthermore, we show theoretical properties of the strengthened QC relaxation and present extensive experimental results that demonstrate the value of the convex relaxation applied in conjunction with OBBT. In particular, we show that 

\begin{enumerate}
    \item The strengthened QC relaxation is able to obtain the tightest voltage and phase angle difference bounds for the AC-OPF problem compared to the other QC relaxations in the literature. 
    \item When utilized in the context of global optimization of AC-OPF with OBBT, on networks with less than 1000 buses, the strengthened QC relaxation results in an optimality gap of $< 1\%$ for 52 out of 57 networks.
\end{enumerate}

\section{The Strengthened QC Relaxation} \label{sec:qc}
This section presents an overview of the mathematical formulation of the AC-OPF and its state-of-the-art QC relaxation and develops the strengthened QC relaxation for the AC-OPF. 

\subsection{AC Optimal Power Flow} \label{subsec:AC-OPF}
We start by presenting the mathematical formulation of the AC-OPF problem in Model \ref{model:ac-opf} with additional $W_{ij}$ and $W_{ii}$ variables for each branch and bus, respectively. The optimal solution to the AC-OPF problem minimizes generation costs for a specified demand and satisfies engineering constraints and power flow physics. 
\begin{model}[t]
\caption{AC Optimal Power Flow (AC-OPF) problem}
\label{model:ac-opf}
\begin{subequations}
\label{eq:AC_polar}
\begin{align}
\mbox{\bf minimize: } & \sum_{i\in \cal{G}} \bm{c}_{2i}(\fr{R}(S_i^g)^2) + \bm{c}_{1i}\fr{R}(S_i^g) + \bm{c}_{0i} \label{objective} \\
\mbox{\bf subject to: } & \nonumber
\end{align}
\vspace{-0.7cm}
\begin{align}
& \sum_{k \in {\cal{G}}_i} S_k^g - \bm{S}_{i}^{\bm{d}} = \sum_{(i,j)\in\cal{E} \cup \cal{E^R}}S_{ij} \quad \forall i \in \cal{N} \label{s_balance} \\
& S_{ij}=\bm{Y}^{*}_{ij}W_{ii} - \bm{Y}^*_{ij}W_{ij} \quad \forall (i,j) \in \cal{E} \label{sij}\\
& S_{ji}=\bm{Y}^{*}_{ij}W_{jj} - \bm{Y}^{*}_{ij}W^*_{ij}\quad \forall (i,j) \in \cal{E} \label{sji}\\
& W_{ii} = |V_i|^2  \quad   \forall i \in \cal{N} \label{wii}\\
& W_{ij} = V_iV_j^* \quad  \forall (i,j) \in \cal{E} \label{wij}\\
& \bm{\theta}^{\bm l}_{ij} \leqslant \theta_{ij} \leqslant \bm{\theta}^{\bm u}_{ij} \quad \forall (i,j) \in \cal{E} \label{theta}\\
& (\bm{v}^{\bm l}_i)^2 \leqslant W_{ii} \leqslant (\bm{v}^{\bm u}_i)^2 \quad  \forall i \in \cal{N} \label{w_limit}\\
& \bm{S}_{i}^{\bm{gl}} \leqslant S^g_i \leqslant \bm{S}_{i}^{\bm{gu}} \quad  \forall i \in \cal{G} \label{g_cap}\\
&|S_{ij}| \leqslant \bm{s}_{ij}^{\bm u} \quad  \forall (i,j) \in \cal{E}\cup \cal{E^R} \label{s_cap}
\end{align}
\end{subequations}
\end{model}
The convex quadratic objective \eqref{objective} minimizes total generation cost. Constraint \eqref{s_balance} enforces nodal power balance at each bus. Constraints \eqref{sij} through \eqref{wij} model the AC power flow on each branch. Constraint \eqref{theta} limits the phase angle difference on each branch. Constraint \eqref{w_limit} limits the voltage magnitude square at each bus. Constraint \eqref{g_cap} restricts the apparent power output of each generator. Finally, constraint \eqref{s_cap} restricts the apparent power transmitted on each branch. For simplicity, we omit the details of constant bus shunt injections, transformer taps, phase shifts, and line charging, though we include them in the computational studies. The AC-OPF is a hard, non-convex problem \cite{Bienstock2015}, with non-convexities arising from the constraints \eqref{wii} and \eqref{wij}.

\subsection{QC Relaxation using Recursive McCormick} \label{subsec:qc1}
The quadratic convex relaxation of the AC-OPF, proposed in \cite{Coffrin2016,Hijazi2017}, is inspired by an arithmetic analysis of \eqref{wii} and \eqref{wij} in polar coordinates (i.e., $V_i = v_i \,\angle \theta_i ~\forall i \in \mathcal N$) with the goal of preserving stronger links between the voltage variables. Rewriting Eq. \eqref{wii} and \eqref{wij} using the polar voltage variables, the non-convexities reduce to the following equations:
\begin{subequations}
\begin{align}
W_{ii} &= v_i^2 \quad \forall i \in \mathcal N\\
\label{eq:wcs}\fr{R}(W_{ij}) &= v_iv_j\cos(\theta_{ij}) \quad \forall (i,j) \in \mathcal E\\
\label{eq:wsn}\fr{I}(W_{ij}) &= v_iv_j\sin(\theta_{ij}) \quad \forall (i,j) \in \mathcal E 
\end{align}
\label{eq:w_nonconv}
\end{subequations}
Each of the above non-convex equations are then relaxed by composing convex envelopes of the non-convex sub-expressions using the bounds on $v_i, v_j, \theta_{ij}$ variables. For the square and product of variables, the QC relaxation uses the well-known McCormick envelopes \cite{McCormick1976}, i.e.,
\begin{equation}
\tag{T-CONV}
\langle x^2 \rangle^T \equiv
\begin{cases}
\widecheck{x}  \geq  x^2\\
\widecheck{x}  \leq  ( \bm {x^u} + \bm {x^l})x - \bm {x^u} \bm {x^l}
\end{cases}
\end{equation}
\begin{equation*}
\tag{M-CONV}
\langle xy \rangle^M \equiv
\begin{cases}
\widecheck{xy}  \geq  \bm {x^l}y + \bm {y^l}x - \bm {x^l}\bm {y^l}\\
\widecheck{xy}  \geq  \bm {x^u}y + \bm {y^u}x - \bm {x^u}\bm {y^u}\\
\widecheck{xy}  \leq  \bm {x^l}y + \bm {y^u}x - \bm {x^l}\bm {y^u}\\
\widecheck{xy}  \leq  \bm {x^u}y + \bm {y^l}x - \bm {x^u}\bm {y^l}
\end{cases}
\end{equation*}
The above convex envelopes are parameterized by the variable bounds (i.e., $\bm {x^l}, \bm {x^u}, \bm {y^l}, \bm {y^u}$). The convex envelopes for the the cosine (C-CONV) and sine (S-CONV) functions, under the assumption that the phase angle difference bound satisfies $- \bm \pi/2 \leq \bm {\theta^l}_{ij} \leq \bm {\theta^u}_{ij} \leq \bm \pi/2$ \cite{Coffrin2017}, are given by 
\begin{equation*}
\langle \cos(x) \rangle^C \equiv
\begin{cases}
\widecheck{cs}  \leq 1 - \frac{1-\cos({\bm {x^m}})}{({\bm {x^m}})^2} x^2\\
\widecheck{cs}  \geq \frac{\cos(\bm {x^l}) - \cos(\bm {x^u})}{(\bm {x^l}-\bm {x^u})}(x - \bm {x^l}) + \cos(\bm {x^l})
\end{cases}
\end{equation*}
\begin{equation*}
\langle \sin(x) \rangle^S \! \equiv \!
\begin{cases}
\widecheck{sn} \leq \cos\left(\frac{\bm {x^m}}{2}\right)\left(x -\frac{\bm {x^m}}{2}\right) + \sin\left(\frac{\bm {x^m}}{2}\right)  \\
\widecheck{sn} \geq \cos\left(\frac{\bm {x^m}}{2}\right)\left(x +\frac{\bm {x^m}}{2}\right) - \sin\left(\frac{\bm {x^m}}{2}\right) \\
\widecheck{sn} \geq \! \frac{\sin(\bm {x^l}) - \sin(\bm {x^u})}{(\bm {x^l}-\bm {x^u})}(x \!-\! \bm {x^l}) \!+\!  \sin(\bm {x^l}) \; \mbox {if} \; \bm {x^l} \! \geq \! 0 \\
\widecheck{sn} \leq \! \frac{\sin(\bm {x^l}) - \sin(\bm {x^u})}{(\bm {x^l}-\bm {x^u})}(x \!-\! \bm {x^l}) \!+\! \sin(\bm {x^l}) \; \mbox {if} \; \bm {x^u} \! \leq \! 0 \\
\end{cases}
\end{equation*}
\noindent
respectively, where $\bm {x^m} = \max(|\bm {x^l}|, |\bm {x^u}|)$. The QC relaxation of the equations \eqref{eq:w_nonconv} is now obtained composing
the convex envelopes for square, sine, cosine, and the product of two variables; the complete relaxation is shown in Model \ref{model:qc-opf}. In Model \ref{model:qc-opf} and the models that follow, we abuse notation and let $\langle f(\cdot)\rangle^C$ denote the variable on the left-hand side of the convex envelope, $C$, for the function $f(\cdot)$. When such an expression is used inside an equation, the constraints $\langle f(\cdot)\rangle^C$ are also added to the model. 
\begin{model}[t]
\allowdisplaybreaks
\caption{Original QC Relaxation (QC-RM).}
\label{model:qc-opf}
\begin{subequations}
\begin{align}
\mbox{\bf minimize: } & \sum_{i\in \cal{G}} \bm{c}_{2i}(\fr{R}(S_i^g)^2) + \bm{c}_{1i}\fr{R}(S_i^g) + \bm{c}_{0i} \\
\mbox{\bf subject to: } & \mbox{\eqref{s_balance} -- \eqref{sji}, \eqref{theta} -- \eqref{s_cap},  \eqref{eq:4d_cut_1} -- \eqref{eq:4d_cut_2}} & \nonumber 
\end{align}
\vspace{-0.7cm}
\begin{align}
& W_{ii} = \langle v_i^2 \rangle^T  \;\; i \in \mathcal N \label{qc_1} \\
&\Re(W_{ij}) = \langle \langle v_i v_j \rangle^M \langle \cos(\theta_{ij}) \rangle^C \rangle^M \;\; \forall(i,j) \in \mathcal E \label{qc_2} \\
&\Im(W_{ij}) = \langle \langle v_i v_j \rangle^M \langle \sin(\theta_{ij}) \rangle^S \rangle^M  \;\; \forall(i,j) \in \mathcal E \label{qc_3} \\
& S_{ij} + S_{ji} = \bm Z_{ij} l_{ij} \;\; \forall (i,j) \in \mathcal E \label{qc_4} \\
& |S_{ij}|^2 \leqslant W_{ii} l_{ij} \;\; \forall (i,j) \in \mathcal E \label{qc_5}
\end{align}
\end{subequations}
\end{model}
Eq. \eqref{qc_4} and \eqref{qc_5} in Model \ref{model:qc-opf} are convex constraints that connect apparent power flow on branches ($S_{ij}$) with current magnitude squared variables ($l_{ij}$).  It is important to highlight that Model \ref{model:qc-opf} includes the ``Lifted Nonlinear Cuts'' (LNCs) of \cite{Coffrin2017}, which further improve the version presented in \cite{Coffrin2016,Hijazi2017}.  The LNCs are formulated using the following constants that are based on variable bounds, i.e.:
\begin{subequations}
\begin{align}
& \bm {v^\sigma}_i = \bm {v^l}_i + \bm {v^u}_i \;\; \forall i \in \mathcal N \label{eq:sec_1} \\
& \bm \phi_{ij} = (\bm {\theta^u}_{ij} + \bm {\theta^l}_{ij})/2   \;\; \forall (i,j)\in \mathcal E \\
& \bm \delta_{ij} = (\bm {\theta^u}_{ij} - \bm {\theta^l}_{ij})/2   \;\; \forall (i,j)\in \mathcal E.
\end{align}
\end{subequations}
The LNCs are then given by \eqref{eq:4d_cut_1}-\eqref{eq:4d_cut_2}, and are linear in the $w_i := W_{ii}, w_j := W_{jj}, w^R_{ij} := \Re(W_{ij}), w^I_{ij} := \mathfrak{I}(W_{ij})$ variables.
%
\begin{figure}[!t]
\hrulefill
\begin{subequations}
\begin{align}
\begin{gathered}
\bm {v^\sigma}_i\bm {v^\sigma}_j(w^R_{ij}\cos(\bm \phi_{ij}) \!+\! w^I_{ij}\sin(\bm \phi_{ij})) \!-\! \bm {v^u}_j \cos(\bm \delta_{ij})\bm {v^\sigma}_jw_i - \bm {v^u}_i \cos(\bm \delta_{ij})\bm {v^\sigma}_iw_j \geqslant \\ 
\bm {v^u}_i\bm {v^u}_j \cos(\bm \delta_{ij})(\bm {v^l}_i\bm {v^l}_j - \bm {v^u}_i\bm {v^u}_j) \; \forall (i,j) \!\in\! \mathcal E
\end{gathered}  \label{eq:4d_cut_1} \\
\begin{gathered}
\bm {v^\sigma}_i\bm {v^\sigma}_j(w^R_{ij}\cos(\bm \phi_{ij}) \!+\! w^I_{ij}\sin(\bm \phi_{ij})) \!-\! \bm {v^l}_j \cos(\bm \delta_{ij})\bm {v^\sigma}_jw_i - \bm {v^l}_i \cos(\bm \delta_{ij})\bm {v^\sigma}_iw_j \geqslant \\ 
\bm {v^l}_i\bm {v^l}_j \cos(\bm \delta_{ij})(\bm {v^u}_i\bm {v^u}_j - \bm {v^l}_i\bm {v^l}_j) \;\; \forall (i,j) \!\in\! \mathcal E
\end{gathered}
\label{eq:4d_cut_2} 
\end{align}
\end{subequations}
\hrulefill
\end{figure}

\subsection{QC Relaxation using Extreme Point Representation}
We now present an alternate QC relaxation that uses an extreme-point representation, instead of applying the McCormick constraints recursively, to express the convex envelope of $\fr{R}(W_{ij})$ and $\fr{I}(W_{ij})$ in Eq. \eqref{eq:w_nonconv}. After the introduction of lifted variables $\widecheck{cs}_{ij}$ and $\widecheck{sn}_{ij}$ for the cosine and sine functions, respectively,  for each branch $(i,j) \in \mathcal E$, the non-convex constraints in Eq. \eqref{eq:wcs} and Eq. \eqref{eq:wsn} become trilinear term of the form $v_i v_j \widecheck{cs}_{ij}$ and $v_i v_j \widecheck{sn}_{ij}$, respectively. This version of the QC relaxation uses the extreme-point representation to obtain the convex envelope of these trilinear terms. It is known in the literature that the extreme-point representation captures the convex hull of a given, single multilinear term \cite{Rikun1997} and that it is tighter than the recursive McCormick envelopes in Eq. \eqref{qc_2} and \eqref{qc_4} \cite{Narimani2018,Nagarajan2018lego}. Nevertheless, though we capture the term-wise convex hull, we lose a potential connection between the voltage products in Eq. \eqref{eq:wcs} and Eq. \eqref{eq:wsn} that is captured in Model \ref{model:qc-opf} using the shared lifted variable, $\widecheck{v_i v_j}$, to capture $\langle v_i v_j \rangle^M$ in Eq. \eqref{qc_2} and Eq. \eqref{qc_3}. Hence, no clear dominance between the original QC relaxation in Model \ref{model:qc-opf} and the QC relaxation using an extreme point representation in the forthcoming Model \ref{model:qc-opf-lambda} can be established. This is also depicted in the Venn diagram in Figure \ref{fig:venn} and later observed in the computational results as well.

We now define an extreme point before describing the convex envelope. Given a set $X$, a point $p \in X$ is extreme if there does not exist two other distinct points $p_1,p_2 \in X$ and a non-negative multiplier $\lambda \in [0,1]$ such that  $p = \lambda p_1 + (1-\lambda)p_2$. To that end, let $\varphi(x_1, x_2, x_3) = x_1 x_2 x_3$ denote a trilinear term with variable bounds $\bm{x^l}_i \leqslant x_i \leqslant \bm{x^u}_i$ for all $i=1,2,3$. Also, let $\bm \xi = \langle \bm{\xi_1}, \dots, \bm{\xi_8}\rangle$ denote the vector of eight extreme points of $[\bm{x^l}_1,\bm{x^u}_1] \times [\bm{x^l}_2,\bm{x^u}_2] \times [\bm{x^l}_3,\bm{x^u}_3]$ and we use $\bm{\xi_k}^i$ to denote the $i$\textsuperscript{th} coordinate of $\bm{\xi_k}$. The extreme points in $\bm \xi$ are given by
\begin{align}
    & \bm{\xi_1} = (\bm{x^l}_1, \bm{x^l}_2, \bm{x^l}_3),\;  \bm{\xi_2} = (\bm{x^l}_1, \bm{x^l}_2, \bm{x^u}_3),\; \bm{\xi_3} = (\bm{x^l}_1, \bm{x^u}_2, \bm{x^l}_3), \nonumber \\
    & \bm{\xi_4} = (\bm{x^l}_1, \bm{x^u}_2, \bm{x^u}_3),\; \bm{\xi_5} = (\bm{x^u}_1, \bm{x^l}_2, \bm{x^l}_3),\;  \bm{\xi_6} = (\bm{x^u}_1, \bm{x^l}_2, \bm{x^u}_3), \nonumber \\
    & \bm{\xi_7} = (\bm{x^u}_1, \bm{x^u}_2, \bm{x^l}_3),\text{ and }  \bm{\xi_8} = (\bm{x^u}_1, \bm{x^u}_2, \bm{x^u}_3). \label{eq:extreme_points}
\end{align}
Then, the tightest convex envelope of the trilinear term $x_1 x_2 x_3$ (TRI-CONV) is given by  
\begin{equation} \label{eq:triform}
\langle x_1 x_2 x_3 \rangle^{\lambda} \equiv
\begin{cases}
\widecheck{x} = \sum_{k=1}^8 \lambda_i \,\varphi(\bm{\xi_k}^1, \bm{\xi_k}^2, \bm{\xi_k}^3) \\
x_i = \sum_{k=1}^8 \lambda_k \, \bm{\xi_k}^i \;\; \forall i=1,2,3 \\
\sum_{k=1}^8 \lambda_k = 1, \;\; \lambda_k \geqslant 0 \;\; \forall k=1,\dots,8
\end{cases}
\end{equation}
Notice, that the lifted variable $\widecheck{x}$ represents the trilinear term i.e., it will replace the the right-hand side of Eq. \eqref{qc_2_lambda} and \eqref{qc_3_lambda}. Using the convex envelope for the trilinear term results in the QC relaxation given by Model \ref{model:qc-opf-lambda}.
\begin{model}[t]
\caption{$\lambda$-based QC relaxation (QC-LM).}
\label{model:qc-opf-lambda}
\begin{subequations}
\begin{align}
\mbox{\bf minimize: } & \sum_{i\in \cal{G}} \bm{c}_{2i}(\fr{R}(S_i^g)^2) + \bm{c}_{1i}\fr{R}(S_i^g) + \bm{c}_{0i} \\
\mbox{\bf subject to: } & \mbox{\eqref{s_balance} -- \eqref{sji}, \eqref{theta} -- \eqref{s_cap}, \eqref{qc_1}, } & \nonumber \\
& \mbox{\eqref{qc_4} -- \eqref{qc_5}, \eqref{eq:4d_cut_1} -- \eqref{eq:4d_cut_2}} & \nonumber
\end{align}
\vspace{-0.7cm}
\begin{align}
&\Re(W_{ij}) = \langle  v_i v_j \widecheck{cs}_{ij} \rangle^{\lambda^c_{ij}} \;\; \forall(i,j) \in \mathcal E \label{qc_2_lambda} \\
&\Im(W_{ij}) = \langle v_i v_j \widecheck{sn}_{ij} \rangle^{\lambda^s_{ij}}  \;\; \forall(i,j) \in \mathcal E \label{qc_3_lambda} 
\end{align}
\end{subequations}
\end{model}
In Model \ref{model:qc-opf-lambda}, the constraints defining the lifted variables $\widecheck{cs}_{ij}$ and $\widecheck{sn}_{ij}$, for each branch $(i,j) \in \mathcal E$, are included in Eq. \eqref{qc_2_lambda} and \eqref{qc_3_lambda}, respectively. We also remark that distinct multiplier variables $\lambda^c_{ij}$ and $\lambda^s_{ij}$ are used for capturing the convex envelopes in Eq. \eqref{qc_2_lambda} and \eqref{qc_3_lambda}, respectively. 

\subsection{A Strengthened QC Relaxation}
This section presents additional constraints to strengthen Model \ref{model:qc-opf-lambda}. The fundamental idea used to develop these strengthening  constraints lies in the observation that different sets of multiplier variables $\lambda^c_{ij}$ and $\lambda^s_{ij}$ are used for capturing the convex envelopes in Eq. \eqref{qc_2_lambda} and \eqref{qc_3_lambda}, respectively. There are no constraints that link $\lambda^c_{ij}$ and $\lambda^s_{ij}$ directly despite sharing two out of three variables in the trilinear term. Adding such a linking constraint intuitively leads to strengthening the relaxation in Model \ref{model:qc-opf-lambda}. We first state the linking constraint for every branch $(i,j) \in \mathcal E$, as follows:
\begin{align}   
    \begin{pmatrix}
        \lambda^c_{ij,1} + \lambda^c_{ij,2} - \lambda^s_{ij,1} - \lambda^s_{ij,2}\\
        \lambda^c_{ij,3} + \lambda^c_{ij,4} - \lambda^s_{ij,3} - \lambda^s_{ij,4} \\
        \lambda^c_{ij,5} + \lambda^c_{ij,6} - \lambda^s_{ij,5} - \lambda^s_{ij,6} \\
        \lambda^c_{ij,7} + \lambda^c_{ij,8} - \lambda^s_{ij,7} - \lambda^s_{ij,8} 
    \end{pmatrix}^T 
    \begin{pmatrix}
    \bm{v^l}_i \cdot \bm{v^l}_j \\
    \bm{v^l}_i \cdot \bm{v^u}_j \\
    \bm{v^u}_i \cdot \bm{v^l}_j \\
    \bm{v^u}_i \cdot \bm{v^u}_j 
    \end{pmatrix} = 0 
    \label{eq:linking}
\end{align}
This constraint enforces, for each branch $(i,j) \in \mathcal E$, the value of the voltage product $v_i v_j$ to take the same value in Eq.  \eqref{qc_2_lambda} and \eqref{qc_3_lambda}. The resulting strengthened QC relaxation is summarized in Model \ref{model:qc-opf-lambda-strengthened}.
\begin{model}
\caption{Tighter $\lambda$-based QC Relaxation (QC-TLM).}
\label{model:qc-opf-lambda-strengthened}
\begin{subequations}
\begin{align}
\mbox{\bf minimize: } & \sum_{i\in \cal{G}} \bm{c}_{2i}(\fr{R}(S_i^g)^2) + \bm{c}_{1i}\fr{R}(S_i^g) + \bm{c}_{0i} \\
\mbox{\bf subject to: } & \mbox{\eqref{s_balance} -- \eqref{sji}, \eqref{theta} -- \eqref{s_cap}, \eqref{qc_1}, \eqref{qc_4} -- \eqref{qc_5}}, & \nonumber \\
& \mbox{\eqref{eq:4d_cut_1} -- \eqref{eq:4d_cut_2}, \eqref{qc_2_lambda} -- \eqref{qc_3_lambda}, \eqref{eq:linking}} & \nonumber
\end{align}
\end{subequations}
\end{model} 

In the following subsection, we detail the theoretical properties of Model \ref{model:qc-opf-lambda-strengthened} and show that it is tighter than the QC relaxations in Model \ref{model:qc-opf} and \ref{model:qc-opf-lambda}.

\subsection{Theoretical Properties of the QC-TLM Relaxation} 
Before presenting the theoretical properties of the strengthened QC relaxation, we first expand constraints in Eq. \eqref{sij}, \eqref{wii}, and \eqref{wij} for a branch $(i,j) \in \mathcal E$ as follows:
\begin{subequations}
\begin{align}
& p_{ij} = \bm{g}_{ij} v_i^2  - \bm{g}_{ij} v_i v_j \cos \theta_{ij} - \bm{b}_{ij} v_i v_j \sin \theta_{ij} \\
& q_{ij} = - \bm{b}_{ij} v_i^2  - \bm{g}_{ij} v_i v_j \cos \theta_{ij} + \bm{b}_{ij} v_i v_j \sin \theta_{ij}
\end{align}
\label{eq:pq}
\end{subequations}
After applying the convex envelopes (C-CONV) and (S-CONV) for the cosine and sine terms, the Eq. \eqref{eq:pq} reduce to
\begin{subequations}
\begin{align}
& p_{ij} = \bm{g}_{ij} v_i^2  - \bm{g}_{ij} v_i v_j \widecheck{cs}_{ij} - \bm{b}_{ij} v_i v_j \widecheck{sn}_{ij} \label{eq:p_sc}\\
& q_{ij} = - \bm{b}_{ij} v_i^2  - \bm{g}_{ij} v_i v_j \widecheck{cs}_{ij} + \bm{b}_{ij} v_i v_j \widecheck{sn}_{ij} \label{eq:q_sc}
\end{align}
\label{eq:pq_sc}
\end{subequations}
Given Eq. \eqref{eq:pq_sc}, the strengthened QC relaxation has the following properties 
\begin{thm} \label{thm:conv_1}
The strengthened QC relaxation in Model \ref{model:qc-opf-lambda-strengthened} captures the convex hull of the nonlinear, non-convex term $\left(-\bm{g}_{ij} v_i v_j \widecheck{cs}_{ij} - \bm{b}_{ij} v_i v_j \widecheck{sn}_{ij}\right)$ in Eq. \eqref{eq:p_sc}.
\end{thm}
\begin{proof} See Sec. \ref{sec:proof}.  \smartqed
\end{proof}
\begin{thm} \label{thm:conv_2}
The strengthened QC relaxation in Model \ref{model:qc-opf-lambda-strengthened} captures the convex hull of the nonlinear, non-convex term $\left(- \bm{g}_{ij} v_i v_j \widecheck{cs}_{ij} + \bm{b}_{ij} v_i v_j \widecheck{sn}_{ij}\right)$ in Eq. \eqref{eq:q_sc}.
\end{thm}
\begin{proof}
See Sec. \ref{sec:proof}. \smartqed
\end{proof}
\begin{figure}[htp]
  \centering
  \includegraphics[scale=0.25]{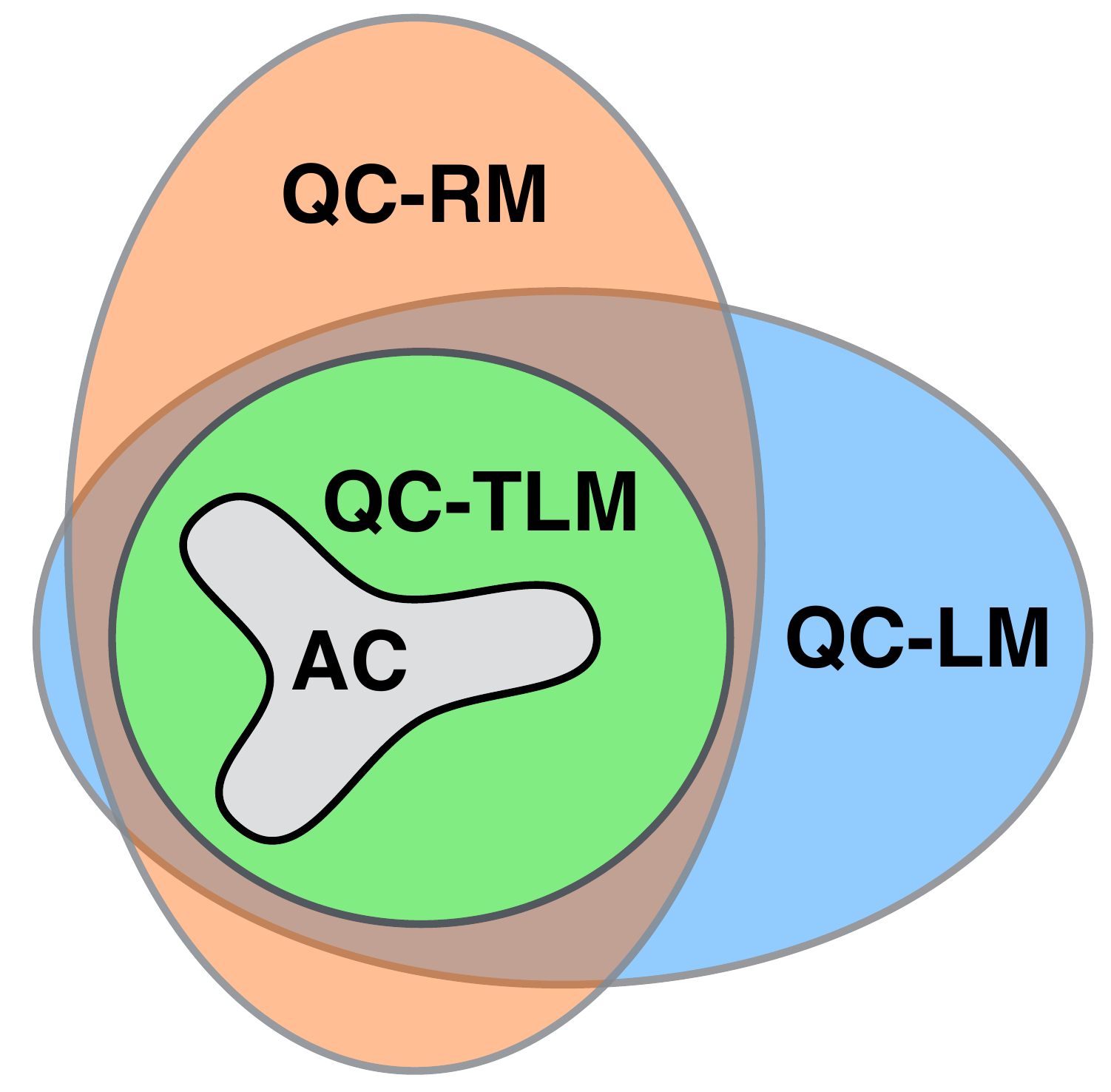}
  \caption{A Venn diagram representing the feasible sets of QC relaxation with various trilinear term relaxations (set sizes in this illustration are not to scale).}
  \label{fig:venn}
\end{figure}

The theoretical properties of the QC relaxations considered here are summarized in Figure \ref{fig:venn}.  
To the best of our knowledge, the theoretical results connecting the summation of multilinear terms presented in this paper are new and novel in the global optimization literature. The computational impact of Theorems \ref{thm:conv_1} and \ref{thm:conv_2} are presented in Section \ref{sec:results}. 

Before, we present the proof of Theorems \ref{thm:conv_1} and \ref{thm:conv_2}, we present some results of a computational experiment comparing the three relaxations namely the Recursive-McCormick (RM), the $\lambda$-based formulation (LM), and the tightened $\lambda$-based formulation (TLM) applied to the sum of two multilinear terms $\phi(x_1, x_2, x_3, x_4) = x_1 x_2 x_3 + x_1 x_2 x_4$ defined on $x_i \in [0, 1]$ for every $i\in \{1, 2, 3, 4\}$. To that end, we randomly generate $5000$ points uniformly in the domain $[0,1]^4$ and for each point $\bm x^k$, calculate the difference between the upper and lower bounds of $\phi(x_1, x_2, x_3, x_4)$ as defined by RM, LM, and TLM formulations. We denote these differences by $\text{RMgap}(\bm x^k)$, $\text{LMgap}(\bm x^k)$, and $\text{TLMgap}(\bm x^k)$, respectively. We then construct two scatter plots, shown in Figure \ref{fig:scatter}, of the points $(\text{LMgap}(\bm x^k), \text{RMgap}(\bm x^k))$ and $(\text{TLMgap}(\bm x^k), \text{RMgap}(\bm x^k))$ for all $k = 1, \dots, 5000$, respectively. 
\begin{figure}[htpb]
\centering{}
\subfloat[RM gap vs. LM gap\label{fig:scatter-1}]{\includegraphics[scale=0.5]{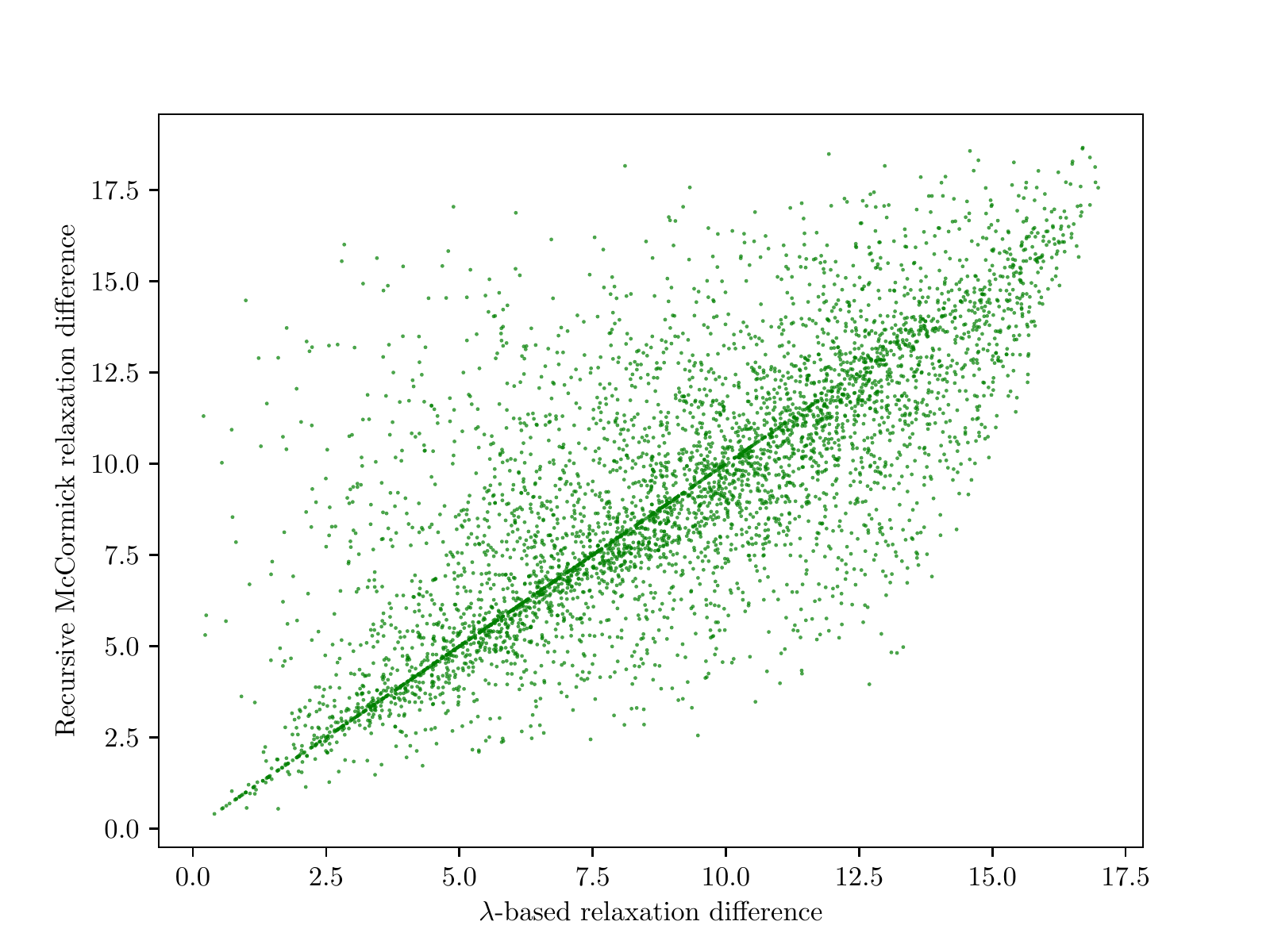}}\\
\subfloat[RM gap vs. TLM gap\label{fig:scatter-2}]{\includegraphics[scale=0.5]{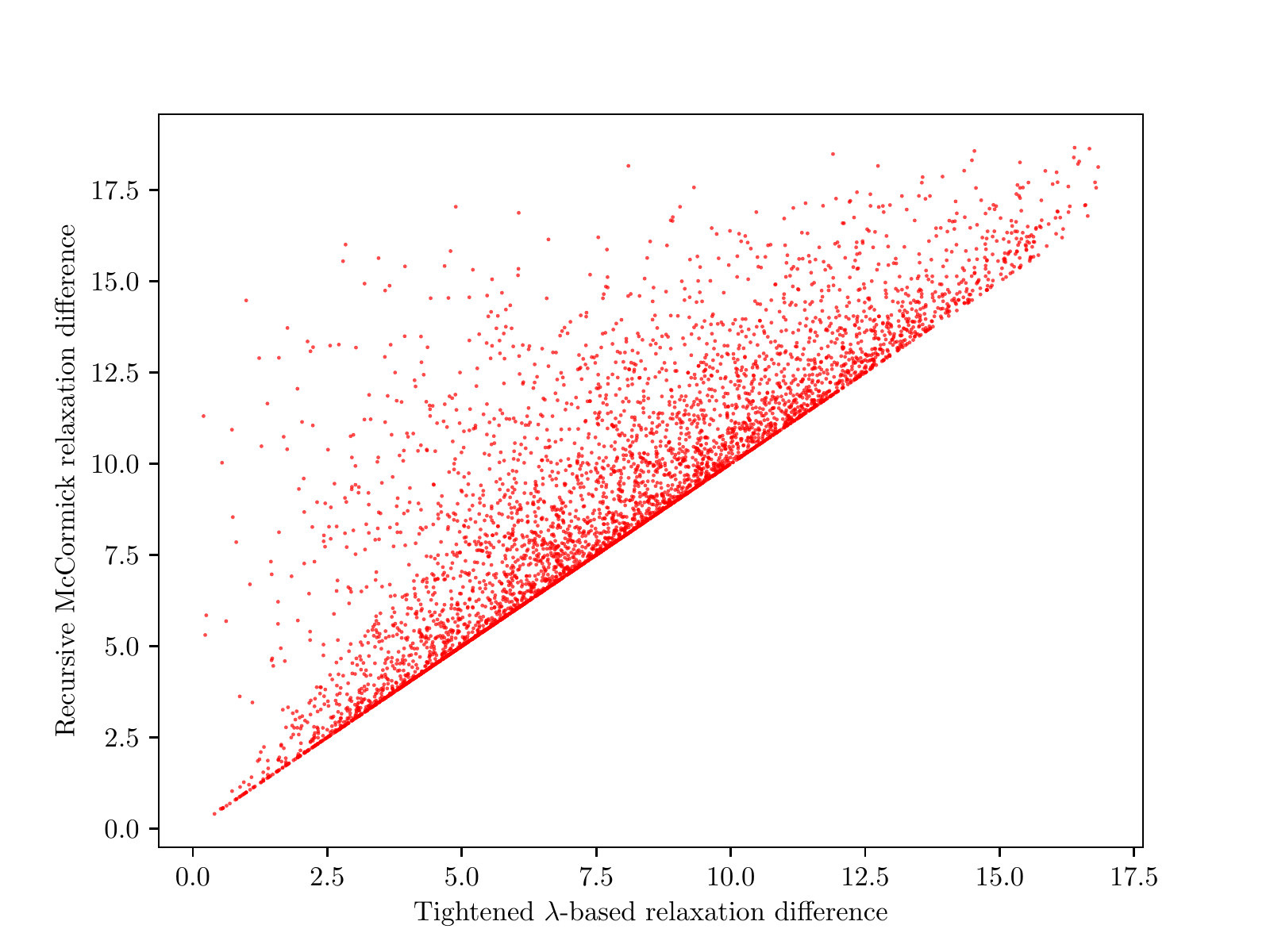}}
\caption{Scatter plots of the gaps obtained using the three relaxations for the sum of trilinear terms $\phi(x_1, x_2, x_3, x_4)$ in the domain $[0,1]^4$.}
\label{fig:scatter}
\end{figure}
The points in scatter plot shown in Fig. \ref{fig:scatter-1}, lie above or below the line of unit slope indicating that there are instances where either relaxation i.e., RM or LM, can be stronger than the other. In contrast, all the points in Fig. \ref{fig:scatter-2} lie above the line of unit slope indicating that the TLM relaxation is always better than RM. Between LM and TLM relaxations, it is clear from the definition of these relaxations that TLM is stronger than LM. In the next section, we present a proof that TLM indeed captures the convex hull of the sum of trilinear terms with two shared variables.  

\subsection{Proof of Theorem \ref{thm:conv_1} and Theorem \ref{thm:conv_2}} \label{sec:proof}
To keep the proof general, we shall present it for the sum of trilinear terms $\phi(x_1, x_2, x_3, x_4) = \bm{\alpha_c} x_1 x_2 x_3 + \bm{\alpha_s} x_1 x_2 x_4$ where $\bm{\alpha_c}, \bm{\alpha_s} \in \mathbb R$ and $\bm{x^l}_i \leqslant x_i \leqslant \bm{x^u}_i$, $i = 1,2,3,4$. The convex hull of $z = \phi(x_1, x_2, x_3, x_4)$ is given by 
\begin{equation} 
\label{eq:cvx_quadform}
\mathcal{S} = 
\begin{cases}
z =  \sum_{k=1}^{16} \lambda_k \phi(\bm{\gamma_k}) \\
x_i = \sum_{k=1}^{16} \lambda_k \, \bm{\gamma_k}^i \;\; \forall i=1,2,3,4 \\
\sum_{k=1}^{16} \lambda_k = 1, \;\; \lambda_k \geqslant 0 \;\; \forall k=1,\dots,16
\end{cases}
\end{equation}

Let $\bm \gamma^c = \langle \bm{\gamma_1^c},\dots,\bm{\gamma_8^c}\rangle$ and $\bm \gamma^s= \langle \bm{\gamma_1^s},\dots,\bm{\gamma_1^s}\rangle$ and $\bm \gamma = \langle \bm{\gamma_1},\dots,\bm{\gamma_{16}}\rangle$ denote the  extreme points of $[\bm{x^l}_1, \bm{x^u}_1] \times [\bm{x^l}_2, \bm{x^u}_2] \times [\bm{x^l}_3, \bm{x^u}_3]$ and $[\bm{x^l}_1, \bm{x^u}_1] \times [\bm{x^l}_2, \bm{x^u}_2] \times [\bm{x^l}_4, \bm{x^u}_4]$, and $[\bm{x^l}_1, \bm{x^u}_1] \times [\bm{x^l}_2, \bm{x^u}_2] \times [\bm{x^l}_3, \bm{x^u}_3]\times [\bm{x^l}_4, \bm{x^u}_4]$, respectively. The extreme points in $\bm \gamma^c$, $\bm \gamma^s$ and $\bm \gamma$ 
are ordered similar to the extreme points in Eq. \eqref{eq:extreme_points}, i.e., in dictionary order. The strengthened QC relaxation represents the term $z = \phi(x_1, x_2, x_3, x_4)$ by using the following equations:
\begin{equation}    
\label{eq:tightQC_general}
\mathcal{S}_{QC} = 
\begin{cases}
    & z_{c} = \langle x_1 x_2 x_3 \rangle^{\lambda^c}, \quad z_{s} = \langle x_1 x_2 x_4 \rangle^{\lambda^s} \\
    & z = \bm{\alpha_c} z_{c} + \bm{\alpha_s} z_{s} \\
    & \text{Eq. \eqref{eq:linking}} \ \mbox{with} \ x_1 \equiv {v_i}, \ x_2 \equiv {v_j}.
\end{cases}
\end{equation}
We show that the projection of \eqref{eq:tightQC_general} on to the $(z,x_1,x_2,x_3,x_4)$ space is identical to its convex hull given in \eqref{eq:cvx_quadform}.

Let $(z,x_1,x_2,x_3,x_4) \in \mathcal{S}$ and let $(\lambda_i) \ i=1,\ldots, 16$ be the corresponding multipliers.
Set 
\begin{subequations} \label{eq:lambda_forward}
\begin{align}   
&\lambda_1^c = \lambda_1 + \lambda_2 \quad &\lambda_1^s &= \lambda_1 + \lambda_3 \\
&\lambda_2^c = \lambda_3 + \lambda_4 \quad &\lambda_2^s &= \lambda_2 + \lambda_4 \\
&\lambda_3^c = \lambda_5 + \lambda_6 \quad &\lambda_3^s &= \lambda_5 + \lambda_7 \\
&\lambda_4^c = \lambda_7 + \lambda_8 \quad &\lambda_4^s &= \lambda_6 + \lambda_8 \\
&\lambda_5^c = \lambda_9 + \lambda_{10} \quad &\lambda_5^s &= \lambda_9 + \lambda_{11} \\
&\lambda_6^c = \lambda_{11} + \lambda_{12} \quad &\lambda_6^s &= \lambda_{10} + \lambda_{12} \\
&\lambda_7^c = \lambda_{13} + \lambda_{14} \quad &\lambda_7^s &= \lambda_{13} + \lambda_{15} \\
&\lambda_8^c = \lambda_{15} + \lambda_{16} \quad &\lambda_8^s &= \lambda_{14} + \lambda_{16}.
\end{align}
\end{subequations}
With the above assignment, it is easy to check that $\lambda_i^c, \lambda_i^s \geq 0$ and $\sum_i \lambda_i^c = \sum_i \lambda_i^s = 1$ and $\bm{\alpha_c} \sum_i \lambda_i^c \gamma_i^c + \bm{\alpha_s} \sum_i \lambda_i^s \gamma_i^s = \sum_i \lambda_i \gamma_i$. This shows that  $(z,x_1,x_2,x_3,x_4) \in \mathcal{P}(\mathcal{S}_{QC})$ and hence that $\mathcal{S} \subseteq \mathcal{P}(\mathcal{S}_{QC})$.

Now, let $(z,z_c,z_s,x_1,x_2,x_3,x_4) \in \mathcal{S}_{QC}$. Let
\begin{subequations}    \label{eq:lambda_backward}
\begin{align}
    \lambda_{1|4} &= \lambda_{\text{odd}}^s - \lambda_{\text{even}}^c + \max\{\lambda_{\text{even}}^c - \lambda_{\text{odd}}^s    , 0\} \\
    \lambda_{2|4} &= \lambda_{\text{even}}^s - \max\{\lambda_{\text{even}}^c - \lambda_{\text{odd}}^s    , 0\} \\
    \lambda_{3|4} &= \lambda_{\text{even}}^c - \max\{\lambda_{\text{even}}^c - \lambda_{\text{odd}}^s    , 0\}  \\
    \lambda_{4|4} &= 0 + \max\{\lambda_{\text{even}}^c - \lambda_{\text{odd}}^s    , 0\}
\end{align}
\end{subequations}
where, $\lambda_{i|4} = (\lambda_i, \lambda_{i+4}, \lambda_{i+8}, \lambda_{i+12})$. Also, $\lambda^s_{\text{even}}$ represents the vector of $\lambda^s$ variables with even indices arranged in sorted order and the vectors $\lambda^s_{\text{odd}}$, $\lambda^c_{\text{even}}$, and $\lambda^c_{\text{odd}}$ are defined in a similar manner.
By construction, the above assignment of $\lambda$ satisfies the system of equations in \eqref{eq:lambda_forward}. As a result, we have $\bm{\alpha_c} \sum_i \lambda_i^c \gamma_i^c + \bm{\alpha_s} \sum_i \lambda_i^s \gamma_i^s = \sum_i \lambda_i \gamma_i$.
Further, $\sum_{i}\lambda_i = \sum_{i} \lambda_i^s = 1$. What is left is to show that $\lambda_i \geq 0 \ \forall i$.
We can rewrite \eqref{eq:lambda_backward} as
\begin{subequations}    \label{eq:lambda_backward_positive}
\begin{align}
    \lambda_{1|4} &= \max\{\lambda_{\text{odd}}^s - \lambda_{\text{even}}^c,0\} \label{eq:l1} \\
    \lambda_{2|4} &=  \min\{\lambda_{\text{even}}^s + \lambda_{\text{odd}}^s  -\lambda_{\text{even}}^c, \lambda_{\text{even}}^s \}  \label{eq:l2} \\
    \lambda_{3|4} &=  \min\{\lambda_{\text{odd}}^s,\lambda_{\text{even}}^c \} \label{eq:l3} \\
    \lambda_{4|4} &= \max\{\lambda_{\text{even}}^c - \lambda_{\text{odd}}^s    , 0\}. \label{eq:l4}
\end{align}
\end{subequations}
We observe that the expressions on the RHS of Equations~\eqref{eq:l1},\eqref{eq:l3} and \eqref{eq:l4} are positive. To show that $\lambda_{2|4}$ in \eqref{eq:l2} is positive, we need the following additional result.
\begin{lem} \label{lem:four_equations_from_one}
The multipliers $\lambda^c$ and $\lambda^s$ satisfy
\begin{align}
    \lambda^c_{\text{even}} +  \lambda^c_{\text{odd}} = \lambda^s_{\text{even}} +  \lambda^s_{\text{odd}}.
\end{align}
\end{lem}
Using Lemma~\ref{lem:four_equations_from_one}, we see that $\lambda_{\text{even}}^s + \lambda_{\text{odd}}^s  -\lambda_{\text{even}}^c = \lambda^c_{\text{odd}}$, and thus $\lambda_{2|4} =  \min\{\lambda^c_{\text{odd}}, \lambda_{\text{even}}^s \} \geq 0$ and proof of the theorem is complete.

\subsubsection{Proof of Lemma~\ref{lem:four_equations_from_one}}
Using Eq~\eqref{eq:triform} and the coupling constraint in \eqref{eq:linking}, we conclude that $\lambda^c,\lambda^s$ satisfy the following constraint.
\begin{align}   \label{eq:unique_cvx}
    \begin{pmatrix}
        \bm{x^l}_1 & \bm{x^l}_2 &  \bm{x^l}_1 \bm{x^l}_2\\
        \bm{x^l}_1 & \bm{x^u}_2 &  \bm{x^l}_1 \bm{x^u}_2\\
        \bm{x^u}_1 & \bm{x^l}_2 &  \bm{x^u}_1 \bm{x^l}_2\\
        \bm{x^u}_1 & \bm{x^u}_2 &  \bm{x^u}_1 \bm{x^u}_2
    \end{pmatrix} 
    \begin{pmatrix}
        \lambda^c_{\text{odd}} +  \lambda^c_{\text{even}} - \lambda^s_{\text{odd}} -  \lambda^s_{\text{even}}
    \end{pmatrix} = 0.
\end{align}
Since the four extreme points appearing in the rows of the matrix in LHS of \eqref{eq:unique_cvx} are assumed to be distinct, the matrix is full rank ($3$) and we must have $ \lambda^c_{\text{odd}} +  \lambda^c_{\text{even}} - \lambda^s_{\text{odd}} -  \lambda^s_{\text{even}} = 0$.

\section{Optimization-Based Bound Tightening} \label{sec:obbt}
We now present an Optimization-Based Bound Tightening (OBBT) algorithm that can be applied to any convex relaxation of the AC-OPF problem with voltage magnitude and phase angle difference variables and is aimed at tightening the bounds on these variables. It has been observed in \cite{Coffrin2015,Chen2016} that the SDP and QC relaxations of AC-OPF  benefit substantially with tight variable bounds. The algorithm proceeds as follows: Let $\Omega$ denote the feasible set of any one of the QC relaxations of the AC-OPF problem presented in this article. Then, two optimization problems, one for each variable in the set $\mathcal V = \{v_i \, \forall i \in \mathcal N, \theta_{ij} \, \forall (i,j) \in \mathcal E\}$ are solved to find the maximum and minimum value of the variable subject to the constraints in $\Omega$. Observe that each optimization problem is convex and upon computing tighter variable bounds for each variable in the set $\mathcal V$, a new, tighter QC relaxation is constructed, if any bound has changed. This process is repeated until a fixed point is reached, i.e., none of the variable bounds change between subsequent iterations. A pseudo-code of the OBBT algorithm is given in Algorithm \ref{alg:obbt}.

\begin{algorithm}
  \caption{The OBBT Algorithm}
  \label{alg:obbt}
  \begin{algorithmic}[1]
    \vspace{1ex}
    \Input A QC Relaxation (Model \ref{model:qc-opf}/\ref{model:qc-opf-lambda}/\ref{model:qc-opf-lambda-strengthened}) to construct $\Omega$
    \Output $\bm{v^l}$, $\bm{v^u}$, $\bm{\theta^l}$, $\bm{\theta^u}$
    \Repeat 
        \State $\bm{v^{l0}}, \bm{v^{u0}}, \bm{\theta^{l0}}, \bm{\theta^{u0}} \gets \bm{v^l}, \bm{v^u}, \bm{\theta^l}, \bm{\theta^u}$
        \State $\Omega \gets$ QC relaxation given $\bm{v^{l0}}, \bm{v^{u0}}, \bm{\theta^{l0}}, \bm{\theta^{u0}}$
        \ForAll{$i \in \mathcal N$}
            \State $\bm{v^l}_i \gets \min\{v_i : \Omega \}$ \label{step:6}
            \State $\bm{v^u}_i \gets \max\{v_i : \Omega \}$ \label{step:7}
        \EndFor
        \ForAll{$(i,j) \in \mathcal E$}
            \State $\bm{\theta^l}_{ij} \gets \min\{\theta_{ij} : \Omega \}$ \label{step:8}
            \State $\bm{\theta^u}_{ij} \gets \max\{\theta_{ij} : \Omega \}$ \label{step:9}
        \EndFor
    \Until{ $\bm{v^{l0}}, \bm{v^{u0}}, \bm{\theta^{l0}}, \bm{\theta^{u0}} = \bm{v^l}, \bm{v^u}, \bm{\theta^l}, \bm{\theta^u}$}
  \end{algorithmic}
\end{algorithm}

\subsection{OBBT for Global Optimization}
The value of using the OBBT algorithm for characterizing the AC-OPF feasibility set was originally highlighted in \cite{Coffrin2015}.  However, recent works have noticed that if the primary goal is to improve the objective lower bound of the AC-OPF problem, then adding the following additional, convex, upper bound constraint to $\Omega$ can vastly improve the algorithm \cite{Lu2018,Liu2018}:
\begin{align}
    \sum_{i\in \cal{G}} \bm{c}_{2i}(\fr{R}(S_i^g)^2) + \bm{c}_{1i}\fr{R}(S_i^g) + \bm{c}_{0i} \leqslant f^* \label{eq:ub-constraint}
\end{align}
where, $f^*$ denotes the cost of any feasible AC-OPF solution or in particular, a local optimal AC-OPF solution. The additional constraint reduces the search space for each convex optimization problem solved during the OBBT algorithm and is routinely used in the global optimization literature \cite{Nagarajan2019,Nagarajan2016}.  We refer to the version of Algorithm \ref{alg:obbt} that includes constraint \eqref{eq:ub-constraint} as GO-OBBT.

\section{Numerical Results}
\label{sec:results}
This section highlights the computational differences of the proposed QC relaxations (i.e. QC-RM, QC-LM, and QC-TLM) via two detailed numerical studies.  The first study revisits the OBBT algorithm from \cite{Coffrin2015} and demonstrates that QC-TLM provides tighter voltage and voltage angle bounds with a negligible change in runtime.  The second study explores the effectiveness of the QC relaxations for providing lower bounds on the AC-OPF both with and without bound tightening.

\subsection{Test Cases and Computational Setting}
This study focuses on 57 networks from the IEEE PES PGLib AC-OPF v18.08 benchmark library \cite{pglib}, which are all under 1000 buses.  Larger cases were not considered due to the computational burden of running the OBBT algorithm on such cases.  All of QC relaxations and the OBBT algorithms were implemented in Julia v0.6 using the optimization modeling layer JuMP.jl v0.18 \cite{Dunning2017}. All of the implementations are available as part of the open-source julia package PowerModels.jl v0.8 \cite{Coffrin2018a}.
Individual non-convex AC-OPF problems and convex QC-OPF relaxations were solved with Ipopt \cite{Ipopt} using the HSL-MA27 linear algebra solver.  The convex relaxations in the OBBT algorithms were solved with Gurobi v8.0 \cite{gurobi} for improved performance and numerical accuracy.  All solvers were set to optimally tolerance of $10^{-6}$ and the OBBT algorithm was configured with a minimum bound width and an average improvement tolerance \cite{Coffrin2018a} of $10^{-3}$ and $10^{-4}$, respectively.  Finally, all of the algorithms were evaluated on HPE ProLiant XL170r servers with two Intel 2.10 GHz CPUs and 128 GB of memory.

\subsection{Computing AC-OPF Feasibility Sets}
Table \ref{tbl:bt_quality} presents the bound improvements from running OBBT (Algorithm \ref{alg:obbt}) with the proposed QC relaxations.  For each network considered, the table reports: (1) the average voltage magnitude bound range (i.e. $\sum_{i \in \mathcal N} (\bm{v^u}_i - \bm{v^l}_i) / |\mathcal N|$); (2) the average voltage angle difference bound range (i.e. $\sum_{(i,j) \in \mathcal E} (\bm{\theta^u}_{ij} - \bm{\theta^l}_{ij}) / |\mathcal E|$) and; (3) the number of branches where the sign of the $\theta$ is fixed (i.e. $\sum_{(i,j) \in \mathcal E} (\bm{\theta^u}_{ij} \leq 0 \vee \bm{\theta^l}_{ij} \geq 0)$).  Bold text is used to highlight the best result in each row of the table.

The results in Table \ref{tbl:bt_quality} highlight the theoretical result showing that the QC-TLM relaxation dominates both the QC-RM and QC-LM relaxations.  It also provides examples where QC-LM dominates QC-RM (e.g. case118\_ieee) and QC-RM dominates QC-LM (e.g. case39\_epri).   It is important to note, that although the differences in average range values may be very small, these lead to significant range reductions when considered across the network.

Figure \ref{fig:bt_rt} presents the distribution of OBBT runtimes for the various QC relaxations presented in Table \ref{tbl:bt_quality}, both in terms of total runtime and an ideal parallel runtime.  These results highlight that there is no significant difference in the runtime of the QC relaxations considered here.   

%
%

Table as of 08/2018
\begin{table*}[h!]
\centering
\footnotesize
\caption{The Quality of QC Relaxations on OBBT Computations.}
\begin{tabular}{|r|r|r||r|r|r||r|r|r||r|r|r|r|r|r|r|r|r|}
\hline
& & & \multicolumn{3}{c||}{Average Vm Range} & \multicolumn{3}{c||}{Average Td Range} & \multicolumn{3}{c|}{Td Sign} \\
Case & $|N|$ & $|E|$ & RM & LM & TLM & RM & LM & TLM & RM & LM & TLM \\
\hline
\hline
\multicolumn{12}{|c|}{Typical Operating Conditions (TYP)} \\
\hline
case3\_lmbd & 3 & 3 & 0.2000 & 0.2000 & 0.2000 & 0.4364 & {\bf 0.4361 } & {\bf 0.4361 } & 2 & 2 & 2 \\
\hline
case5\_pjm & 5 & 6 & 0.1981 & 0.1981 & 0.1981 & 0.0718 & 0.0716 & {\bf 0.0714 } & 3 & 3 & 3 \\
\hline
case14\_ieee & 14 & 20 & 0.0883 & 0.0883 & 0.0883 & 0.0165 & 0.0166 & {\bf 0.0164 } & 18 & 18 & 18 \\
\hline
case24\_ieee\_rts & 24 & 38 & 0.0895 & 0.0895 & 0.0895 & 0.1067 & 0.1063 & {\bf 0.1062 } & 19 & 19 & 19 \\
\hline
case30\_as & 30 & 41 & 0.0771 & 0.0771 & 0.0771 & 0.0294 & 0.0294 & {\bf 0.0293 } & 31 & 31 & {\bf 32 } \\
\hline
case30\_fsr & 30 & 41 & 0.0927 & 0.0927 & 0.0927 & 0.0403 & 0.0402 & {\bf 0.0401 } & 9 & 9 & 9 \\
\hline
case30\_ieee & 30 & 41 & 0.0587 & 0.0587 & 0.0587 & 0.0064 & 0.0064 & 0.0064 & 36 & 36 & 36 \\
\hline
case39\_epri & 39 & 46 & 0.1058 & 0.1058 & 0.1058 & 0.0885 & 0.0887 & {\bf 0.0884 } & 21 & 21 & 21 \\
\hline
case57\_ieee & 57 & 80 & 0.0358 & 0.0358 & 0.0358 & 0.0501 & 0.0501 & {\bf 0.0500 } & 42 & 42 & 42 \\
\hline
case73\_ieee\_rts & 73 & 120 & 0.0911 & 0.0911 & 0.0911 & 0.1772 & 0.1761 & {\bf 0.1758 } & 39 & 39 & 39 \\
\hline
case89\_pegase & 89 & 210 & 0.0951 & 0.0950 & {\bf 0.0949 } & 0.0805 & 0.0806 & {\bf 0.0798 } & 107 & 107 & {\bf 108 } \\
\hline
case118\_ieee & 118 & 186 & 0.1068 & {\bf 0.1067 } & {\bf 0.1067 } & 0.1144 & 0.1136 & {\bf 0.1133 } & 70 & {\bf 71 } & {\bf 71 } \\
\hline
case162\_ieee\_dtc & 162 & 284 & 0.0645 & 0.0642 & {\bf 0.0641 } & 0.0419 & 0.0417 & {\bf 0.0414 } & 230 & {\bf 231 } & {\bf 231 } \\
\hline
case179\_goc & 179 & 263 & 0.1828 & 0.1828 & 0.1828 & 0.1635 & 0.1634 & {\bf 0.1628 } & 65 & 65 & 65 \\
\hline
case200\_tamu & 200 & 245 & 0.1906 & 0.1906 & 0.1906 & 0.0300 & 0.0302 & {\bf 0.0297 } & 124 & 121 & {\bf 126 } \\
\hline
case240\_pserc & 240 & 448 & 0.1919 & {\bf 0.1918 } & {\bf 0.1918 } & 0.2350 & 0.2328 & {\bf 0.2307 } & 81 & 84 & {\bf 85 } \\
\hline
case300\_ieee & 300 & 411 & 0.0713 & 0.0713 & 0.0713 & 0.1481 & 0.1478 & {\bf 0.1476 } & 140 & 140 & 140 \\
\hline
case500\_tamu & 500 & 597 & 0.1755 & 0.1755 & {\bf 0.1754 } & 0.0216 & 0.0218 & {\bf 0.0215 } & 375 & 373 & {\bf 376 } \\
\hline
case588\_sdet & 588 & 686 & 0.1844 & 0.1844 & {\bf 0.1843 } & 0.0602 & 0.0603 & {\bf 0.0599 } & {\bf 174 } & 173 & {\bf 174 } \\
\hline
\hline
\multicolumn{12}{|c|}{Congested Operating Conditions (API)} \\
\hline
case3\_lmbd\_api & 3 & 3 & 0.0379 & 0.0380 & {\bf 0.0378 } & {\bf 0.0464 } & 0.0465 & 0.0465 & 3 & 3 & 3 \\
\hline
case5\_pjm\_api & 5 & 6 & 0.0485 & 0.0485 & 0.0485 & 0.0271 & {\bf 0.0270 } & {\bf 0.0270 } & 4 & 4 & 4 \\
\hline
case14\_ieee\_api & 14 & 20 & 0.0416 & 0.0415 & {\bf 0.0412 } & 0.0135 & 0.0135 & {\bf 0.0134 } & 19 & 19 & 19 \\
\hline
case24\_ieee\_rts\_api & 24 & 38 & 0.0485 & {\bf 0.0484 } & {\bf 0.0484 } & 0.1122 & 0.1119 & {\bf 0.1118 } & 24 & 24 & 24 \\
\hline
case30\_as\_api & 30 & 41 & 0.0100 & 0.0100 & 0.0100 & 0.0067 & 0.0067 & 0.0067 & 39 & 39 & 39 \\
\hline
case30\_fsr\_api & 30 & 41 & 0.0363 & 0.0361 & {\bf 0.0358 } & 0.0140 & 0.0140 & {\bf 0.0138 } & 30 & 30 & 30 \\
\hline
case30\_ieee\_api & 30 & 41 & 0.0206 & 0.0205 & {\bf 0.0203 } & 0.0037 & 0.0037 & 0.0037 & 39 & 39 & 39 \\
\hline
case39\_epri\_api & 39 & 46 & 0.0385 & 0.0386 & {\bf 0.0384 } & 0.0165 & 0.0165 & 0.0165 & 43 & 43 & 43 \\
\hline
case57\_ieee\_api & 57 & 80 & 0.0237 & 0.0237 & 0.0237 & 0.0588 & {\bf 0.0586 } & {\bf 0.0586 } & 41 & 41 & 41 \\
\hline
case73\_ieee\_rts\_api & 73 & 120 & 0.0508 & {\bf 0.0507 } & {\bf 0.0507 } & 0.1660 & 0.1654 & {\bf 0.1653 } & 65 & 65 & 65 \\
\hline
case89\_pegase\_api & 89 & 210 & 0.1322 & 0.1325 & {\bf 0.1307 } & 0.0841 & 0.0849 & {\bf 0.0798 } & 100 & 99 & {\bf 101 } \\
\hline
case118\_ieee\_api & 118 & 186 & 0.1088 & 0.1080 & {\bf 0.1076 } & 0.0778 & 0.0749 & {\bf 0.0738 } & 127 & 129 & {\bf 130 } \\
\hline
case162\_ieee\_dtc\_api & 162 & 284 & 0.0304 & 0.0304 & {\bf 0.0301 } & 0.0079 & 0.0079 & {\bf 0.0078 } & 273 & 273 & 273 \\
\hline
case179\_goc\_api & 179 & 263 & 0.1699 & 0.1699 & 0.1699 & 0.1459 & 0.1457 & {\bf 0.1454 } & 108 & 108 & 108 \\
\hline
case200\_tamu\_api & 200 & 245 & 0.1877 & 0.1877 & 0.1877 & 0.0926 & 0.0930 & {\bf 0.0921 } & 73 & 73 & 73 \\
\hline
case240\_pserc\_api & 240 & 448 & 0.1747 & 0.1746 & {\bf 0.1745 } & 0.2490 & 0.2469 & {\bf 0.2461 } & 92 & 92 & 92 \\
\hline
case300\_ieee\_api & 300 & 411 & 0.0830 & 0.0830 & 0.0830 & 0.1808 & 0.1803 & {\bf 0.1799 } & 131 & 131 & 131 \\
\hline
case500\_tamu\_api & 500 & 597 & 0.1789 & 0.1789 & 0.1789 & 0.0583 & 0.0586 & {\bf 0.0580 } & 265 & 265 & {\bf 267 } \\
\hline
case588\_sdet\_api & 588 & 686 & 0.1835 & 0.1835 & 0.1835 & 0.0705 & 0.0704 & {\bf 0.0702 } & 160 & 160 & {\bf 161 } \\
\hline
\hline
\multicolumn{12}{|c|}{Small Angle Difference Conditions (SAD)} \\
\hline
case3\_lmbd\_sad & 3 & 3 & 0.0947 & 0.0947 & 0.0947 & 0.0701 & 0.0701 & 0.0701 & 2 & 2 & 2 \\
\hline
case5\_pjm\_sad & 5 & 6 & 0.0483 & {\bf 0.0482 } & {\bf 0.0482 } & 0.0062 & 0.0062 & 0.0062 & 5 & 5 & 5 \\
\hline
case14\_ieee\_sad & 14 & 20 & 0.0540 & 0.0540 & 0.0540 & 0.0069 & 0.0069 & 0.0069 & 19 & 19 & 19 \\
\hline
case24\_ieee\_rts\_sad & 24 & 38 & 0.0762 & 0.0762 & {\bf 0.0761 } & 0.0296 & 0.0296 & {\bf 0.0295 } & 29 & 29 & 29 \\
\hline
case30\_as\_sad & 30 & 41 & 0.0464 & 0.0464 & 0.0464 & 0.0086 & 0.0086 & {\bf 0.0085 } & 39 & 39 & 39 \\
\hline
case30\_fsr\_sad & 30 & 41 & 0.0574 & 0.0574 & 0.0574 & 0.0215 & 0.0215 & {\bf 0.0214 } & 22 & 22 & 22 \\
\hline
case30\_ieee\_sad & 30 & 41 & 0.0462 & 0.0462 & 0.0462 & 0.0045 & 0.0045 & 0.0045 & 36 & 36 & 36 \\
\hline
case39\_epri\_sad & 39 & 46 & 0.0701 & 0.0700 & {\bf 0.0699 } & 0.0209 & 0.0209 & 0.0209 & 39 & 39 & 39 \\
\hline
case57\_ieee\_sad & 57 & 80 & 0.0320 & 0.0320 & 0.0320 & 0.0093 & 0.0093 & {\bf 0.0092 } & 71 & 71 & 71 \\
\hline
case73\_ieee\_rts\_sad & 73 & 120 & 0.0843 & {\bf 0.0842 } & {\bf 0.0842 } & 0.0473 & 0.0472 & {\bf 0.0471 } & 78 & 78 & 78 \\
\hline
case89\_pegase\_sad & 89 & 210 & 0.0917 & 0.0917 & 0.0917 & 0.0611 & 0.0612 & {\bf 0.0605 } & 113 & 113 & 113 \\
\hline
case118\_ieee\_sad & 118 & 186 & 0.1043 & 0.1042 & {\bf 0.1041 } & 0.0713 & 0.0707 & {\bf 0.0706 } & 103 & 103 & 103 \\
\hline
case162\_ieee\_dtc\_sad & 162 & 284 & 0.0552 & 0.0550 & {\bf 0.0549 } & 0.0315 & 0.0314 & {\bf 0.0312 } & 238 & 238 & 238 \\
\hline
case179\_goc\_sad & 179 & 263 & 0.1505 & 0.1506 & {\bf 0.1503 } & 0.0943 & 0.0942 & {\bf 0.0938 } & 77 & 77 & 77 \\
\hline
case200\_tamu\_sad & 200 & 245 & 0.0921 & 0.0921 & {\bf 0.0920 } & 0.0191 & 0.0191 & {\bf 0.0190 } & 143 & 143 & 143 \\
\hline
case240\_pserc\_sad & 240 & 448 & 0.1843 & 0.1840 & {\bf 0.1832 } & 0.1717 & 0.1707 & {\bf 0.1685 } & 128 & 128 & {\bf 134 } \\
\hline
case300\_ieee\_sad & 300 & 411 & 0.0693 & 0.0693 & 0.0693 & 0.1428 & 0.1426 & {\bf 0.1424 } & 142 & 142 & 142 \\
\hline
case500\_tamu\_sad & 500 & 597 & 0.1040 & 0.1046 & {\bf 0.1039 } & 0.0106 & 0.0106 & {\bf 0.0105 } & 428 & 428 & 428 \\
\hline
case588\_sdet\_sad & 588 & 686 & 0.1708 & 0.1706 & {\bf 0.1701 } & 0.0378 & 0.0379 & {\bf 0.0375 } & 214 & 214 & 214 \\
\hline
\end{tabular}\\
\label{tbl:bt_quality}
\end{table*}


%
%

\begin{figure}[htp]
  \centering
  \includegraphics[width=8.8cm]{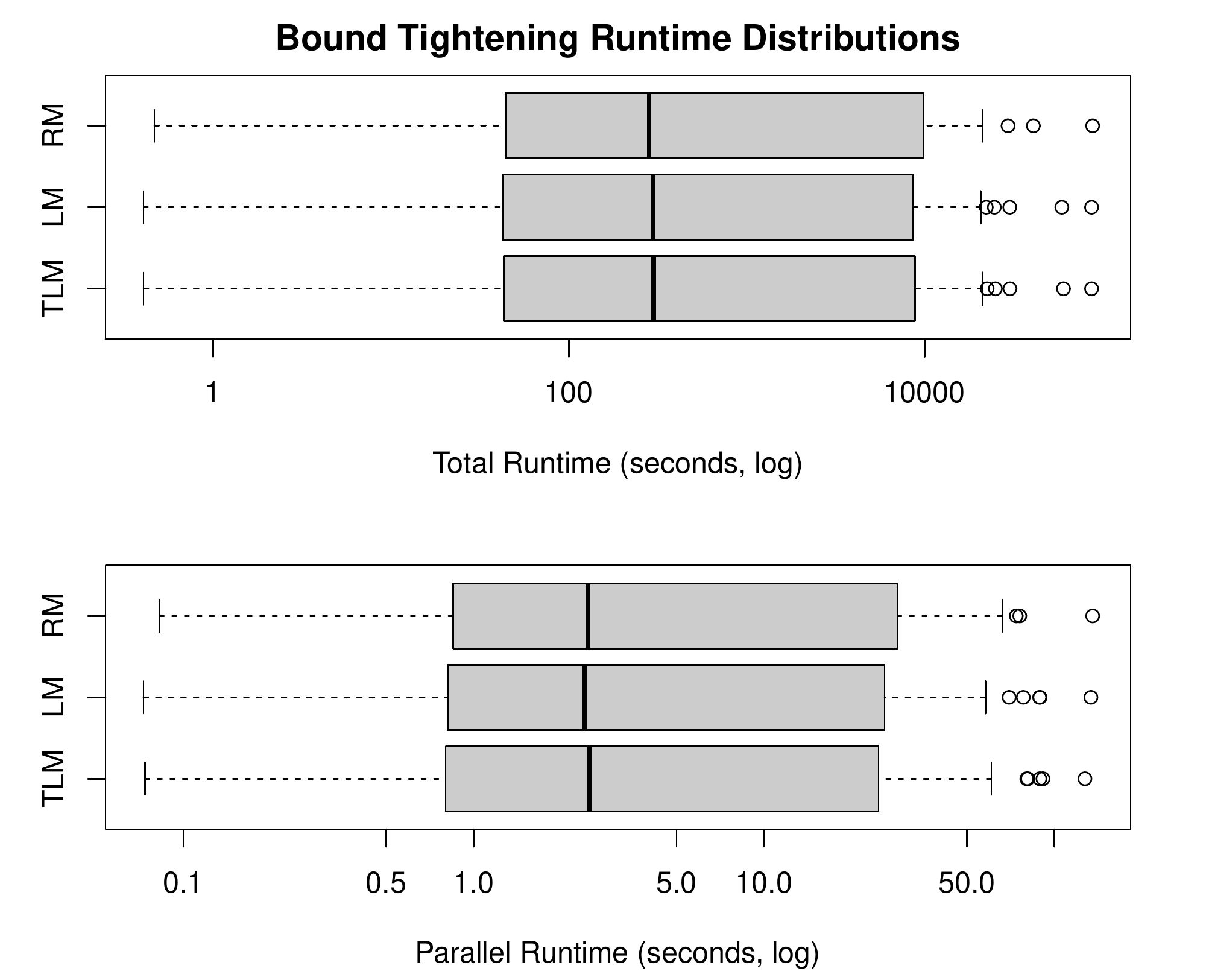}
  \caption{Runtime distributions of OBBT with various QC relaxations. The lower and upper ends of the boxes reflect the first and third quartiles, the lines inside the boxes denote the median, and the circles are outliers. }
  \label{fig:bt_rt}
\end{figure}


\subsection{Computing AC-OPF Lower Bounds}

Table \ref{tbl:go_quality} presents AC-OPF optimality gaps provided by the proposed QC relaxations both with and without bound tightening, where the percentage gap is defined as $100*(AC Heuristic - Relaxation)/AC Heuristic$.  For each network considered, the table reports: (1) The AC objective value from solving the non-convex problem with Ipopt; (2) the optimally gap of each relaxation, without bound tightening; (3) the optimally gap of each relaxation after running GO-OBBT.  In the interest of brevity, cases where the base QC-RM gap is $<1.0\%$ are omitted.  Bold text is used to highlight the best result in each row of the table.

Again, the results in Table \ref{tbl:go_quality} highlight the theoretical result showing that the QC-TLM relaxation dominates both the QC-RM and QC-LM relaxations.  It also highlights two interesting points: (1) In some cases the QC-TLM relaxation can provide benefits without bound tightening, e.g. case24\_ieee\_rts\_api, case73\_ieee\_rts\_api and case14\_ieee\_sad; (2) the QC-TLM relaxation's most significant benefits occur in the most challenging GO-OBBT cases, e.g. case89\_pegase\_api, case118\_ieee\_api.  It is important to note that small discrepancies in the GO-OBBT optimality gaps are observed.  These are due to numerical challenges resulting from the finite precision of floating point arithmetic and only occur in cases where the optimality gap is close to zero.


%
%

\begin{table*}[t]
\centering
\footnotesize
\caption{The Quality of QC Relaxations for AC-OPF Lower Bounds.}
\begin{tabular}{|r|r|r||r||r|r|r||r|r|r|r|r|r|r|r|r|r|r|r|}
\hline
& & & & \multicolumn{3}{c||}{Base Opt. Gap (\%)} & \multicolumn{3}{c|}{GO-OBBT Opt. Gap (\%)} \\
Case & $|N|$ & $|E|$ & AC Obj. & RM & LM & TLM & RM & LM & TLM \\
\hline
\hline
\multicolumn{10}{|c|}{Typical Operating Conditions (TYP)} \\
\hline
case3\_lmbd & 3 & 3 & 5.8126e+03 & 1.22 & {\bf 0.97 } & {\bf 0.97 } & 0.01 & 0.01 & 0.01 \\
\hline
case5\_pjm & 5 & 6 & 1.7552e+04 & 14.55 & 14.55 & 14.55 & 6.01 & 6.14 & {\bf 5.80 } \\
\hline
case30\_ieee & 30 & 41 & 1.1974e+04 & 10.78 & {\bf 10.67 } & {\bf 10.67 } & 0.01 & 0.01 & 0.01 \\
\hline
case118\_ieee & 118 & 186 & 1.1580e+05 & 2.20 & {\bf 2.18 } & {\bf 2.18 } & 0.02 & 0.02 & 0.02 \\
\hline
case162\_ieee\_dtc & 162 & 284 & 1.2615e+05 & 7.54 & 7.54 & 7.54 & 0.05 & {\bf 0.03 } & 0.04 \\
\hline
case240\_pserc & 240 & 448 & 3.5700e+06 & 3.81 & 3.80 & {\bf 3.79 } & 2.37 & 2.36 & {\bf 2.30 } \\
\hline
case300\_ieee & 300 & 411 & 6.6422e+05 & 2.56 & {\bf 2.54 } & {\bf 2.54 } & {\bf 0.06 } & {\bf 0.06 } & 0.07 \\
\hline
case500\_tamu & 500 & 597 & 7.2578e+04 & 5.39 & 5.39 & 5.39 & 0.01 & 0.01 & 0.01 \\
\hline
case588\_sdet & 588 & 686 & 3.8155e+05 & 1.68 & 1.68 & 1.68 & 0.33 & 0.35 & {\bf 0.32 } \\
\hline
\hline
\multicolumn{10}{|c|}{Congested Operating Conditions (API)} \\
\hline
case3\_lmbd\_api & 3 & 3 & 1.1242e+04 & 5.63 & {\bf 4.58 } & {\bf 4.58 } & 0.04 & 0.04 & 0.04 \\
\hline
case5\_pjm\_api & 5 & 6 & 7.6377e+04 & 4.09 & 4.09 & 4.09 & 0.01 & 0.01 & 0.01 \\
\hline
case14\_ieee\_api & 14 & 20 & 1.3311e+04 & 1.77 & 1.77 & 1.77 & 0.02 & {\bf 0.01 } & 0.02 \\
\hline
case24\_ieee\_rts\_api & 24 & 38 & 1.3495e+05 & 13.01 & 11.06 & {\bf 11.03 } & 0.04 & 0.04 & 0.04 \\
\hline
case30\_as\_api & 30 & 41 & 4.9962e+03 & 44.61 & 44.61 & 44.61 & {\bf 0.72 } & 0.77 & 0.80 \\
\hline
case30\_fsr\_api & 30 & 41 & 7.0115e+02 & 2.76 & 2.76 & 2.76 & 0.13 & 0.13 & 0.13 \\
\hline
case30\_ieee\_api & 30 & 41 & 2.4032e+04 & 3.73 & 3.73 & 3.73 & 0.04 & 0.04 & 0.04 \\
\hline
case39\_epri\_api & 39 & 46 & 2.5721e+05 & 1.57 & 1.57 & 1.57 & 0.02 & 0.02 & 0.02 \\
\hline
case73\_ieee\_rts\_api & 73 & 120 & 4.2273e+05 & 11.07 & 9.56 & {\bf 9.54 } & {\bf 0.41 } & {\bf 0.41 } & 0.46 \\
\hline
case89\_pegase\_api & 89 & 210 & 1.4198e+05 & 8.13 & 8.13 & 8.13 & 1.69 & 1.50 & {\bf 1.33 } \\
\hline
case118\_ieee\_api & 118 & 186 & 3.1642e+05 & 28.63 & {\bf 28.62 } & {\bf 28.62 } & 4.27 & 3.64 & {\bf 3.39 } \\
\hline
case162\_ieee\_dtc\_api & 162 & 284 & 1.4351e+05 & 5.44 & 5.44 & 5.44 & {\bf 0.06 } & {\bf 0.06 } & 0.07 \\
\hline
case179\_goc\_api & 179 & 263 & 2.1326e+06 & 7.18 & 7.21 & {\bf 7.10 } & 0.03 & {\bf 0.02 } & {\bf 0.02 } \\
\hline
\hline
\multicolumn{10}{|c|}{Small Angle Difference Conditions (SAD)} \\
\hline
case3\_lmbd\_sad & 3 & 3 & 5.9593e+03 & 1.42 & {\bf 1.38 } & {\bf 1.38 } & 0.03 & 0.03 & 0.03 \\
\hline
case14\_ieee\_sad & 14 & 20 & 6.7834e+03 & 7.16 & 6.38 & {\bf 6.36 } & 0.30 & 0.30 & 0.30 \\
\hline
case24\_ieee\_rts\_sad & 24 & 38 & 7.6943e+04 & 2.93 & 2.77 & {\bf 2.74 } & 0.23 & 0.23 & 0.23 \\
\hline
case30\_as\_sad & 30 & 41 & 8.9749e+02 & 2.32 & 2.32 & {\bf 2.31 } & {\bf 0.31 } & 0.32 & 0.32 \\
\hline
case30\_ieee\_sad & 30 & 41 & 1.1974e+04 & 3.42 & 3.28 & {\bf 3.24 } & 0.01 & 0.01 & 0.01 \\
\hline
case73\_ieee\_rts\_sad & 73 & 120 & 2.2775e+05 & 2.54 & 2.39 & {\bf 2.38 } & {\bf 0.09 } & {\bf 0.09 } & 0.10 \\
\hline
case118\_ieee\_sad & 118 & 186 & 1.2924e+05 & 9.48 & 9.31 & {\bf 9.30 } & {\bf 0.24 } & 0.25 & 0.26 \\
\hline
case162\_ieee\_dtc\_sad & 162 & 284 & 1.2704e+05 & 8.02 & 7.98 & {\bf 7.97 } & 0.08 & 0.08 & 0.08 \\
\hline
case179\_goc\_sad & 179 & 263 & 8.3560e+05 & 1.05 & {\bf 1.04 } & {\bf 1.04 } & 0.02 & 0.02 & 0.02 \\
\hline
case240\_pserc\_sad & 240 & 448 & 3.6565e+06 & 5.24 & 5.22 & {\bf 5.21 } & 2.83 & 2.82 & {\bf 2.70 } \\
\hline
case300\_ieee\_sad & 300 & 411 & 6.6431e+05 & 2.36 & 2.30 & {\bf 2.29 } & 0.04 & 0.04 & 0.04 \\
\hline
case500\_tamu\_sad & 500 & 597 & 7.9234e+04 & 7.90 & 7.90 & 7.90 & 0.31 & 0.34 & {\bf 0.30 } \\
\hline
case588\_sdet\_sad & 588 & 686 & 4.0427e+05 & 6.26 & 6.28 & {\bf 6.24 } & 0.25 & 0.26 & {\bf 0.24 } \\
\hline
\end{tabular}\\
\label{tbl:go_quality}
\end{table*}

\section{Conclusion} \label{sec:conclusion}
In summary, this article presents a strengthened version of the QC relaxation and shows its theoretical tightness and its effectiveness in computing better variable bounds and reducing the optimality gap on a wide range of test networks, when used in conjunction with bound-tightening techniques. Future research directions include extensions of these relaxations to the optimal transmission switching problem, both theoretically and computationally.


\bibliographystyle{plain}
\bibliography{references.bib}

LA-UR-18-28769

\end{document}